\documentclass[12pt]{amsart}
\usepackage[margin=1in]{geometry}
\usepackage{amsmath,amssymb,amsthm,mathtools}

\newtheorem{theorem}{Theorem}[section]
\newtheorem{lemma}[theorem]{Lemma}
\newtheorem{proposition}[theorem]{Proposition}
\newtheorem{corollary}[theorem]{Corollary}
\theoremstyle{definition}
\newtheorem{definition}[theorem]{Definition}

\theoremstyle{remark}
\newtheorem{remark}[theorem]{Remark}
\newtheorem{conjecture}[theorem]{Conjecture}
\numberwithin{equation}{section}
\newcommand{\D}{\mathbb{D}}
\newcommand{\T}{\mathbb{T}}
\newcommand{\C}{\mathbb{C}}
\newcommand{\R}{\mathbb{R}}

\newcommand{\dd}{\,dm}
\newcommand{\Pp}{\mathcal{P}_+}
\DeclareMathOperator{\arctanh}{arctanh}
\DeclareMathOperator{\sinc}{sinc}
\DeclareMathOperator{\sgn}{sgn}
\DeclareMathOperator{\Real}{Re}
\DeclareMathOperator{\Imag}{Im}

\title{Schwarz-contractivity of the Poisson operator}

\author{Leonid V. Kovalev}
\address{215 Carnegie, Department of Mathematics, Syracuse University, Syracuse, NY 13244, USA \\
ORCID: 0000-0001-8002-7155}
\email{lvkovale@syr.edu}



\begin{document}

\begin{abstract}
We say that the Poisson operator $P_r$ is \textit{Schwarz-contractive} from a
space $X$ to a space $Y$ if $\|P_r\|_{X\to Y}\le r$ holds for all $0<r<1$. The classical Schwarz lemma is the case $H^\infty_0\to H^\infty$, the subscript indicating zero mean. We are concerned with the non-holomorphic case $L^\infty_0 \to L^p$. For real functions Schwarz-contractivity holds up to the sharp exponent $p_{\mathbb R}=4.109\ldots$. For complex functions we prove it for $p\le 3$ and conjecture that the critical exponent is $4$. 
\end{abstract}

\maketitle

\section{Introduction}\label{sec:intro}

Let $\D=\{z\in\C\colon |z|<1\}$ and $\T=\partial\D$, and let $dm=d\theta/(2\pi)$
be the normalized measure on $\T$, so that all norms $\|\cdot\|_p$ below are
taken with respect to a probability measure. For $f\in L^1(\T)$ and $0<r<1$
let $P_rf$ denote the Poisson integral of $f$ evaluated on the circle of
radius $r$:
\begin{equation}\label{eq-poisson}
P_rf(\zeta)=\int_\T K_r(\zeta\overline\eta)f(\eta)\dd(\eta),
\qquad
K_r(e^{i\theta})=\frac{1-r^2}{1-2r\cos\theta+r^2},
\end{equation}
equivalently $\widehat{P_rf}(n)=r^{|n|}\widehat f(n)$ for $n\in\mathbb Z$.
Every bounded harmonic function on $\D$ is the Poisson integral of its
boundary values, and we use this identification throughout. We write
$L^\infty_0$ for the subspace of $L^\infty(\T)$ consisting of functions with
$\widehat f(0)=0$, and similarly $H^\infty_0$ for bounded holomorphic
functions vanishing at the origin.

\begin{definition}\label{def-schwarz-contractive}
The Poisson operator is \textit{Schwarz-contractive} from $X$ to $Y$ if
$\|P_rf\|_Y\le r\|f\|_X$ for all $f\in X$ and all $0<r<1$; that is, if
$\|P_r\|_{X\to Y}\le r$ for every radius.
\end{definition}

The classical Schwarz lemma is exactly the statement that $P_r$ is
Schwarz-contractive from $H^\infty_0$ to $H^\infty$, with equality for
$f(\zeta)=\zeta$. Harmonic functions do not obey it. By a theorem of
Heinz~\cite{Heinz}, see also Kalaj and Vuorinen~\cite[Lemma 1.4]{KalajVuorinen},
a harmonic map $u\colon\D\to\D$ with $u(0)=0$ satisfies
$|u(z)|\le\frac{4}{\pi}\arctan|z|$, and the bound is attained at every point. Thus, Schwarz-contractivity fails into $L^\infty$, by a factor that tends to $4/\pi$ as $r\to0$. 

However, the norm bound is restored once the target
$L^\infty$ is replaced by $L^p$ with $p$ not too large (clearly, $p=2$ works), which raises the question of how large $p$ may be. The present paper answers it for real functions and gives two-sided bounds for complex ones. 

Put
\begin{equation}\label{eq-Cp}
C_p=\int_\T\left|\frac4\pi\cos\theta\right|^p\dd
=\left(\frac4\pi\right)^p
\frac{\Gamma\left(\frac{p+1}2\right)}{\sqrt\pi\,\Gamma\left(\frac p2+1\right)},
\qquad p>0.
\end{equation}
The quantity $C_p^{1/p}=(4/\pi)\|\cos\theta\|_p$ increases strictly with $p$,
and $C_4=96/\pi^4<1$ while $C_5=16384/(15\pi^6)>1$. Hence there is a unique
$p_{\mathbb R}$ with $C_{p_{\mathbb R}}=1$; it lies in $(4,5)$ and equals
\begin{equation}\label{eq-pstar}
p_{\mathbb R}=4.1095286347363\ldots
\end{equation} 

\begin{theorem}\label{thm-real}
The Poisson operator is Schwarz-contractive from $L^\infty_0(\T;\R)$ to $L^p(\T)$ if and only if $p\le p_{\mathbb R}$.
\end{theorem}

Here and below $L^p(\T;\R)$ denotes Lebesgue spaces of real-valued functions, $L^p(\T)$ stands for spaces of complex-valued functions, and the subscript $0$ indicates the zero-mean subspace. The exponent 
$p_{\mathbb R}$ does not carry over to the complex case: see 
Proposition~\ref{prop-complex-obstruction}.

\begin{theorem}\label{thm-complex}
The Poisson operator is Schwarz-contractive from $L^\infty_0(\T)$ to $L^3(\T)$.
\end{theorem}

Theorem~\ref{thm-complex} is proved in \S\ref{sec:complex}. It does not assert that $3$ is optimal. If $p_\C$ denotes the supremum of the exponents for which the Poisson operator is Schwarz-contractive from $L^\infty_0(\T)$ to $L^p$, then
\begin{equation}\label{eq-pc-bounds}
3\le p_\C\le 4.
\end{equation}
In \S\ref{sec:questions} we conjecture that $p_\C=4$.

Without the restriction $\widehat f(0)=0$, the norm of $P_r$ from
$L^\infty$ to $L^p$ is $1$ for every $0<p\le\infty$: the Poisson kernel
is nonnegative with integral $1$, and $P_r1=1$ gives equality. But one can
consider the centered Poisson operator $\widetilde P_rf=P_rf-\widehat f(0)$ instead.

\begin{theorem}\label{thm-centered}
The centered Poisson operator is Schwarz-contractive from $L^\infty(\T;\R)$ to $L^p(\T)$ if and only if $p\le 3$.
Equivalently, every harmonic function $u\colon \D\to [-1, 1]$ satisfies
\[
\int_\T|u(r\zeta)-u(0)|^3\dd(\zeta)\le r^3
\]
where the exponent $3$ is best possible.
\end{theorem}

Theorem~\ref{thm-real} is proved in Section~\ref{sec:real}
via Littlewood's subordination principle and monotonicity arguments with respect to $r$. 
Section~\ref{sec:centered} proves Theorem~\ref{thm-centered} by representing the Poisson kernel as a positive mixture of arc averages. In Section~\ref{sec:complex}, Theorem~\ref{thm-complex} is proved using a theorem of Marzo and Seip~\cite{MarzoSeip} and a recent breakthrough by Fu and Li~\cite{FuLi}.

\section{Real functions with zero mean}\label{sec:real}

\subsection{Reduction by  subordination}\label{sec:reduction}

The conformal map $\Phi(z)=(4/\pi)\arctan z$ takes $\D$ onto the strip
$\{w\colon |\Real w|<1\}$ and fixes the origin. Its real part is
\begin{equation}\label{eq-U}
U(re^{i\theta})=\frac2\pi\arctan\frac{2r\cos\theta}{1-r^2},
\end{equation}
the Poisson integral of $\sgn\cos\theta$. On the positive axis it simplifies to
$U(r)=\frac{4}{\pi}\arctan r$.

\begin{proposition}\label{prop-subordination}
Let $1\le p<\infty$ and $0<r<1$. Then
\[
\sup\left\{\|P_rf\|_p\colon f\in L^\infty_0(\T;\R),\ \|f\|_\infty\le1\right\}
=\|P_r(\sgn\cos\theta)\|_p=\|U(r\,\cdot)\|_p .
\]
\end{proposition}

\begin{proof}
Let $u=P[f]$ with $f$ as above. The maximum principle gives $|u|<1$, so the
analytic completion $F=u+iv$ with $F(0)=0$ maps $\D$ into the strip
$\{|\Real w|<1\}$. Hence $\omega=\Phi^{-1}\circ F$ is a holomorphic self-map
of $\D$ with $\omega(0)=0$, and $u=U\circ\omega$. Thus $u$ is subordinate to
$U$ in the sense of Littlewood~\cite[p. 10]{DurenHp}.

Since $x\mapsto|x|^p$ is convex for $p\ge1$ and $U$ is harmonic, $|U|^p$ is
subharmonic, and $|\omega(z)|\le|z|$ by Schwarz's lemma. Littlewood's
subordination theorem in its subharmonic form, inequality~(6) in the proof
of~\cite[Theorem 1.7]{DurenHp}, therefore gives
\begin{equation}\label{eq-subordination}
\int_\T|u(r\zeta)|^p\dd(\zeta)\le\int_\T|U(r\zeta)|^p\dd(\zeta).
\end{equation}
Since the function $f(\theta)=\sgn\cos\theta$ is admissible in the supremum, equality  is attained.
\end{proof}

It is convenient to rescale. For $0<s<1$ put
\begin{equation}\label{eq-Bp}
Y_s(\theta)=\frac{\left|U(\sqrt s\,e^{i\theta})\right|}{\sqrt s},
\qquad
B_p(s)=\int_\T Y_s^p\dd .
\end{equation}
The function $B_p$ extends continuously
to $[0,1]$, with
\begin{equation}\label{eq-endpoints}
B_p(0)=C_p,\qquad B_p(1)=1,
\end{equation}
because $U(re^{i\theta})/r\to(4/\pi)\cos\theta$ uniformly as $r\to0$, while
$|U(re^{i\theta})|\to1$ for almost every $\theta$ as $r\to1$.
With $s=r^2$,
Proposition~\ref{prop-subordination} says that
$\|P_r\|_{L^\infty_0(\T;\R)\to L^p}=r\,B_p(s)^{1/p}$ for $p\ge1$. We record
the consequence.

\begin{corollary}\label{cor-reduction}
For $1\le p<\infty$, the Poisson operator is Schwarz-contractive from
$L^\infty_0(\T;\R)$ to $L^p$ if and only if $B_p\le1$ on $[0,1]$. In
particular this fails for every $p>p_{\mathbb R}$, since then $B_p(0)=C_p>1$.
\end{corollary}

Two further properties of $Y_s$ will be used. Writing
\begin{equation}\label{eq-Yform}
Y_s(\theta)=\frac{2}{\pi\sqrt s}\arctan\left(t\,|\cos\theta|\right),
\qquad t=\frac{2\sqrt s}{1-s},
\end{equation}
we see that $Y_s$ is a fixed increasing function of $|\cos\theta|$. Its maximum is at $\theta=0$ and admits the
integral representation
\begin{equation}\label{eq-sup}
\phi(s):=\|Y_s\|_\infty=\frac4\pi\,\frac{\arctan\sqrt s}{\sqrt s}
=\frac4\pi\int_0^1\frac{dx}{1+sx^2},
\end{equation}
from which $\phi$ is decreasing on $[0,1]$, with $\phi(0)=4/\pi$
and $\phi(1)=1$.

\subsection{A comparison of moment ratios}\label{sec:ratio}

By~\eqref{eq-Yform} all the functions $Y_s$ arise from the single function
$g=|\cos\theta|$ by applying $x\mapsto\arctan(tx)$ and rescaling. As $t$
grows, this map compresses the upper range of $g$, so the higher moments of
$Y_s$ grow more slowly than those of $g$. The next lemma is the quantitative
form of this observation.

\begin{lemma}\label{lem-ratio}
Let $g\ge0$ be a bounded measurable function on $\T$, not almost everywhere
zero, and let $p\ge q>0$. Then
\[
t\mapsto\frac{\|\arctan(tg)\|_p}{\|\arctan(tg)\|_q}
\]
is nonincreasing on $(0,\infty)$, and hence is at most its limiting value
$\|g\|_p/\|g\|_q$ as $t\downarrow0$.
\end{lemma}

\begin{proof}
Write $G=G_t=\arctan(tg)$ and set $\kappa(0)=1$ and
\[
\kappa(x)=\frac{x}{(1+x^2)\arctan x},\qquad x>0 .
\]
Then $\kappa$ is strictly decreasing, since
\[
\kappa'(x)=\frac{(1-x^2)\arctan x-x}{(1+x^2)^2(\arctan x)^2}<0 :
\]
indeed $(1-x^2)\arctan x\le\arctan x<x$ for $0<x<1$, and
$(1-x^2)\arctan x\le0<x$ for $x\ge1$. Because
$t\,\partial_tG=tg/(1+t^2g^2)=G\,\kappa(tg)$, differentiating under the
integral sign (legitimate since $g$ is bounded) gives
\begin{equation}\label{eq-logderiv}
t\frac{d}{dt}\log\|G\|_p=\frac{\int_\T G^p\,\kappa(tg)\dd}{\int_\T G^p\dd}.
\end{equation}
The functions $G^{p-q}$ and $\kappa(tg)$ are, respectively, nondecreasing
and nonincreasing functions of $g$, so for all $\theta,\eta\in\T$,
\[
\bigl(G(\theta)^{p-q}-G(\eta)^{p-q}\bigr)
\bigl(\kappa(tg(\theta))-\kappa(tg(\eta))\bigr)\le0 .
\]
Multiplying by $G(\theta)^qG(\eta)^q\ge0$, integrating over
$\T\times\T$ and expanding the product, we obtain
\[
2\left(\int_\T G^p\kappa(tg)\dd\int_\T G^q\dd
-\int_\T G^p\dd\int_\T G^q\kappa(tg)\dd\right)\le0,
\]
that is,
\[
\frac{\int_\T G^p\kappa(tg)\dd}{\int_\T G^p\dd}
\le\frac{\int_\T G^q\kappa(tg)\dd}{\int_\T G^q\dd},
\]
the denominators being positive because $g$ is not almost everywhere zero.
By~\eqref{eq-logderiv} this says that
$t\frac{d}{dt}\log\left(\|G\|_p/\|G\|_q\right)\le0$. Since
$\arctan(tg)/t\to g$ uniformly as $t\downarrow0$, the limiting value of the
ratio is $\|g\|_p/\|g\|_q$.
\end{proof}

\begin{corollary}\label{cor-comparison}
For $p\ge4$ and $0\le s\le1$,
\begin{equation}\label{eq-comparison}
B_p(s)\le C_p\left(\frac{B_4(s)}{C_4}\right)^{p/4}.
\end{equation}
\end{corollary}

\begin{proof}
Apply Lemma~\ref{lem-ratio} with $g=|\cos\theta|$, $q=4$ and
$t=2\sqrt s/(1-s)$. By~\eqref{eq-Yform} the factor $2/(\pi\sqrt s)$ cancels
in the ratio, so
\[
\frac{\|Y_s\|_p}{\|Y_s\|_4}\le\frac{\|\cos\theta\|_p}{\|\cos\theta\|_4}
=\frac{C_p^{1/p}}{C_4^{1/4}}
\]
by~\eqref{eq-Cp}. Raise to the power $p$.
\end{proof}

Thanks to~\eqref{eq-comparison}, we can obtain $B_p\le C_p$ from the simpler inequality $B_4\le C_4$ when the latter holds.  

\subsection{The fourth moment near the origin}\label{sec:quartic}

Let $L(z)=\arctanh z$ and
\[
h_m=\sum_{j=1}^m\frac1{2j-1},\qquad
t_m=\sum_{j=1}^m\frac{h_j}{j},\qquad h_0=t_0=0 .
\]
Since $L'(z)=(1-z^2)^{-1}$, we have $(L^2)'=2L/(1-z^2)$ and
$(L^3)'=3L^2/(1-z^2)$. Multiplication by
$(1-z^2)^{-1}=\sum_{k\ge0}z^{2k}$ replaces the coefficients of a power series
by their partial sums along even steps, so
\[
\frac{L(z)}{1-z^2}=\sum_{m\ge1}h_mz^{2m-1},
\qquad
\frac{L(z)^2}{1-z^2}=\sum_{m\ge1}t_mz^{2m},
\]
and integrating term by term yields
\begin{equation}\label{eq-LL}
L(z)^2=\sum_{m\ge1}\frac{h_m}{m}z^{2m},
\qquad
L(z)^3=\sum_{m\ge1}\frac{3t_m}{2m+1}z^{2m+1}.
\end{equation}

\begin{lemma}\label{lem-series} With $B_4$ as in~\eqref{eq-Bp} and $C_4=96/\pi^4$ we have, for all $0\le s<1$,
\begin{equation}\label{eq-series}
\frac{B_4(s)}{C_4}=\sum_{m\ge1}\left(\frac{h_m^2}{m^2}s^{2m-2}
-\frac{4t_m}{(2m+1)^2}s^{2m-1}\right).
\end{equation}
\end{lemma}

\begin{proof}
From $\arctan z=-iL(iz)$ we get $\Real\arctan z=\Imag L(iz)$, hence
$U=(4/\pi)\Imag L(i\,\cdot)$ by~\eqref{eq-U}. Rotation invariance of $m$
gives $\int_\T U(r\zeta)^4\dd=(4/\pi)^4\int_\T(\Imag L(r\zeta))^4\dd$. For
analytic $W$ with $W(0)=0$ one has $\int_\T W^4\dd=0$, so expanding
$(\Imag W)^4=\frac1{16}(W-\overline W)^4$ leaves
\[
\int_\T(\Imag W)^4\dd=\frac38\int_\T|W|^4\dd
-\frac12\Real\int_\T W^3\overline W\dd .
\]
Apply this to $W=L(r\,\cdot)$ and use~\eqref{eq-LL} with Parseval's
identity:
\[
\int_\T|L(r\zeta)|^4\dd=\sum_{m\ge1}\frac{h_m^2}{m^2}r^{4m},
\qquad
\Real\int_\T L(r\zeta)^3\overline{L(r\zeta)}\dd
=3\sum_{m\ge1}\frac{t_m}{(2m+1)^2}r^{4m+2}.
\]
With $s=r^2$ and $C_4=\frac38 \cdot (4/\pi)^4$ this gives
\[ \frac{s^2B_4(s)}{C_4} = \sum_{m \ge 1}  \frac{h_m^2}{m^{2}} s^{2m}-4\sum_{m\ge 1} \frac{t_m}{(2m+1)^{2}}s^{2m+1},\] which
is~\eqref{eq-series}.
\end{proof}

Setting $s=0$ in~\eqref{eq-series} recovers $B_4(0)=C_4$, and letting
$s\to1$ recovers $B_4(1)=1$. The only consequence of~\eqref{eq-series} that
we need is the following.

\begin{lemma}\label{lem-left}
$B_4(s)\le C_4$ for $0\le s\le\frac78$.
\end{lemma}

\begin{proof}
The term $m=1$ of the first sum in~\eqref{eq-series} equals $h_1^2=1$.
Pairing each negative term with the positive term following it gives
\begin{equation}\label{eq-regroup}
\frac{B_4(s)}{C_4}-1=\sum_{m\ge1}s^{2m-1}
\left[\frac{h_{m+1}^2}{(m+1)^2}\,s-\frac{4t_m}{(2m+1)^2}\right].
\end{equation}
We claim that every bracket in~\eqref{eq-regroup} is nonpositive when $s\le 7/8$. Since 
\[
\frac{2m+1}{2m+2}h_{m+1}
 =h_m+\frac{1-h_m}{2m+2}\le h_m
\]
it suffices to prove
\begin{equation}\label{eq-tm}
\frac{7}{8} h_m^2 \le t_m, \quad m\ge1.
\end{equation}
Put $E_m=t_m-\frac78h_m^2$. The defining recurrences give
\[
E_m-E_{m-1}
=\frac{(m-4)h_m}{4m(2m-1)}+\frac7{8(2m-1)^2}\ge0,
\qquad m\ge4.
\]
Since $E_1=\frac18$, $E_2=\frac19$ and $E_3=\frac{217}{1800}$ are
positive,~\eqref{eq-tm} follows.  
\end{proof}

\subsection{The second moment near the boundary}\label{sec:boundary}

As $s\to1$ the comparison~\eqref{eq-comparison} becomes useless, since its
right side tends to $C_pC_4^{-p/4}>1$. Near the boundary we interpolate
between $L^2$ and $L^\infty$ instead, where both endpoints are explicit.

\begin{lemma}\label{lem-B2}
For $0\le s\le1$,
\[
B_2(s)=\frac8{\pi^2}\sum_{k\ge0}\frac{s^{2k}}{(2k+1)^2}.
\]
In particular $B_2$ has nonnegative Taylor coefficients, so it is increasing
and convex, with
$B_2(0)=C_2=8/\pi^2$ and $B_2(1)=1$.
\end{lemma}

\begin{proof}
The Fourier coefficients of $\sgn\cos\theta$ are $2(-1)^k/(\pi(2k+1))$ at the
frequencies $\pm(2k+1)$ and zero otherwise, so Parseval's identity gives
$\int_\T U(r\zeta)^2\dd=\frac8{\pi^2}\sum_{k\ge0}r^{2(2k+1)}(2k+1)^{-2}$.
Divide by $s=r^2$ to obtain $B_2(s)$. 
\end{proof}

\begin{lemma}\label{lem-right} With $\phi$ as in~\eqref{eq-sup}, for $\frac78\le s\le1$ we have $\phi(s)^3B_2(s)\le1$, hence $B_{p_{\mathbb R}}(s)\le1$.
\end{lemma}

\begin{proof}
Put $a=\frac78$ and $u=1-s\in[0,\frac18]$. Since $\phi(1)=1$,
\eqref{eq-sup} gives
\[
\phi(s)-1=\frac{4u}\pi\int_0^1
\frac{x^2\,dx}{(1+sx^2)(1+x^2)}.
\]
For $s\ge a$, convexity of $y\mapsto(1+ay)^{-1}$ gives
\[
\frac1{1+sx^2}\le\frac1{1+ax^2}\le1-\frac7{15}x^2,
\qquad 0\le x\le1.
\]
Using $\int_0^1x^2/(1+x^2)\,dx=1-\pi/4$ and
$\int_0^1x^4/(1+x^2)\,dx=\pi/4-2/3$, we obtain
\begin{equation}\label{eq-phibound}
\phi(s)-1\le\left(\frac{236}{45\pi}-\frac{22}{15}\right)u
\le\frac5{24}u,
\end{equation}
where the last inequality follows from $\pi>157/50$.

Four terms of the series in Lemma~\ref{lem-B2}, followed by a geometric
bound for the tail, give
\begin{equation}\label{eq-B2value}
B_2(a)\le\frac8{\pi^2}\left[
\sum_{k=0}^3\frac{a^{2k}}{(2k+1)^2}
+\frac{a^8}{81(1-a^2)}\right]
<\frac8{\pi^2}\cdot\frac{142}{125}
<\frac{59}{64}.
\end{equation}
Convexity of $B_2$ and $B_2(1)=1$ now imply
\begin{equation}\label{eq-B2bound}
1-B_2(s)\ge\frac{1-s}{1-a}\bigl(1-B_2(a)\bigr)\ge\frac58u.
\end{equation}
Consequently, $\log x\le x-1$ yields the exact cancellation
\[
\log\bigl(\phi(s)^3B_2(s)\bigr)
\le3(\phi(s)-1)-(1-B_2(s))
\le\frac58u-\frac58u=0.
\]
Finally, since $\phi(s)\ge1$ and $4<p_{\mathbb R}<5$,
\[
B_{p_{\mathbb R}}(s)
\le\phi(s)^{p_{\mathbb R}-2}B_2(s)
\le\phi(s)^3B_2(s)\le1.
\qedhere \]
\end{proof}

\begin{proof}[Proof of Theorem~\ref{thm-real}]
By Corollary~\ref{cor-reduction} it suffices to prove $B_{p_{\mathbb R}}\le1$ on
$[0,1]$. For $0\le s\le\frac78$, Lemma~\ref{lem-left} gives $B_4(s)\le C_4$,
so the comparison~\eqref{eq-comparison} with $p=p_{\mathbb R}$, together with
$C_{p_{\mathbb R}}=1$, yields
\[
B_{p_{\mathbb R}}(s)\le C_{p_{\mathbb R}}\left(\frac{B_4(s)}{C_4}\right)^{p_{\mathbb R}/4}\le1 .
\]
For $\frac78\le s\le1$ this is Lemma~\ref{lem-right}. Hence
$\|P_rf\|_{p_{\mathbb R}}\le r\|f\|_\infty$ for real $f$ of mean zero, and the same
holds for $0<p\le p_{\mathbb R}$ because $L^p$ norms increase with $p$ on a probability
space.

The failure for $p>p_{\mathbb R}$ is part of Corollary~\ref{cor-reduction}.
\end{proof}

\section{Real functions with nonzero mean}\label{sec:centered}

Theorem~\ref{thm-real} requires $\widehat f(0)=0$. For arbitrary real $f$
with $\|f\|_\infty\le1$ one may of course apply it to $f-\widehat f(0)$, but
$\|f-\widehat f(0)\|_\infty$ will be in general greater than $1$. Theorem~\ref{thm-centered} requires a different argument, which goes as follows: the Poisson kernel is a positive mixture of normalized arc kernels; each arc average admits an $L^3$ estimate by the first
Fourier coefficient of its kernel; and the first Fourier coefficient of $K_r$ is exactly $r$. 

For $0<t\le\pi$ set
\[
\chi_t=\frac\pi t\mathbf 1_{\{|\theta|<t\}},
\qquad
A_tf=\chi_t*f,\quad\text{that is}\quad
A_tf(\theta)=\frac1{2t}\int_{-t}^tf(\theta+\sigma)\,d\sigma,
\]
the convolution being normalized as in~\eqref{eq-poisson}. Then
$\int_\T\chi_t\dd=1$. Using the sinc function $\sinc t = (\sin t)/t$ we can write $\widehat{A_tf}(n)=\sinc(nt)\widehat f(n)$;
in particular $\widehat{\chi_t}(\pm1)=\sinc t$.

The arc estimate follows from a sharp bound for periodic Lipschitz functions.

\begin{lemma}\label{lem-lipschitz}
Let $0<t\le\pi/2$, and let $h$ be a real-valued $2\pi$-periodic Lipschitz
function with $|h|\le1$ and $|h'|\le1/t$ almost everywhere. Then 
\begin{equation}\label{eq-lipschitz}
\|h-\widehat h(0)\|_3^3\le1-\frac{3t}{2\pi}.
\end{equation}
\end{lemma}

\begin{proof}
Put $a=\widehat h(0)$.
Let $[c-b,c+b]$ with $0\le b\le1$ be the range of $h$. If $b=0$ then
$h\equiv a$ and there is nothing to prove, so assume $b>0$ and put
\[
g=\frac{h-c}{b},\qquad\mu=\int_\T g\dd,
\]
so that $g$ takes values in $[-1,1]$, attains both endpoints,
$|g'|\le(bt)^{-1}$, and $h-a=b(g-\mu)$.

\textit{Step 1: a lower bound for integrals of $\psi(g)$.} We claim that
for every continuous $\psi\ge0$ on $[-1,1]$,
\begin{equation}\label{eq-uniform-component}
\int_\T\psi(g)\dd \ge \frac{bt}\pi \int_{-1}^1\psi(x)\,dx.
\end{equation}
Indeed, choose $\theta_-,\theta_+\in\T$ with $g(\theta_\pm)=\pm1$. These
points are distinct, so they split $\T$ into two closed arcs $I_1,I_2$ with
disjoint interiors, each having $\theta_-$ as one endpoint and $\theta_+$ as
the other. (The function $g$ need not be monotone on $I_j$, and it may attain
$\pm1$ at other points as well; only its values at the endpoints of $I_j$
will matter.) If $\Psi'=\psi$ then $\Psi\circ g$ is Lipschitz with
$(\Psi\circ g)'=\psi(g)g'$ almost everywhere, so by the fundamental theorem
of calculus on $I_j$,
\[
\int_{I_j}\psi(g)|g'|\,d\theta
\ge\left|\int_{I_j}\psi(g)g'\,d\theta\right|
=|\Psi(1)-\Psi(-1)|=\int_{-1}^1\psi(x)\,dx .
\]
Since $|g'|\le(bt)^{-1}$, each arc contributes
$\int_{I_j}\psi(g)\,d\theta\ge bt\int_{-1}^1\psi$; adding the two and
dividing by $2\pi$ gives~\eqref{eq-uniform-component}.

\textit{Step 2: the chord estimate.} Let $\varphi(x)=|x-\mu|^3$ and let
\[
\ell(x)=\frac{1-x}2(1+\mu)^3+\frac{1+x}2(1-\mu)^3
\]
be the affine function agreeing with $\varphi$ at $\pm1$. By convexity
$\ell-\varphi\ge0$ on $[-1,1]$, and since $\ell$ is affine,
$\int_\T\ell(g)\dd=\ell(\mu)=1-\mu^4$. Applying~\eqref{eq-uniform-component}
to $\psi=\ell-\varphi$ gives
\[
\int_\T|g-\mu|^3\dd=\ell(\mu)-\int_\T(\ell-\varphi)(g)\dd
\le 1-\mu^4- \frac{bt}\pi\int_{-1}^1(\ell-\varphi)\,dx .
\]
Here $\int_{-1}^1\ell\,dx=2+6\mu^2$ and
$\int_{-1}^1\varphi\,dx=(1+6\mu^2+\mu^4)/2$, so
\[
\int_\T|g-\mu|^3\dd\le1- \frac{3bt}{2\pi} - \frac{3bt}\pi \mu^2-\left(1-\frac{bt}{2\pi}\right)\mu^4
\le1-\frac{3bt}{2\pi}.
\]

\textit{Step 3: conclusion.} Since $h-a=b(g-\mu)$,
\[
\|h-a\|_3^3\le b^3\left(1-\frac{3bt}{2\pi}\right)\le1-\frac{3t}{2\pi},
\]
the last step because $b\mapsto b^3-\frac{3t}{2\pi}b^4$ has derivative
$3b^2\left(1- 2tb/\pi\right)\ge0$ for $b\in[0,1]$ when $t\le\pi/2$.
\end{proof}

\begin{lemma}\label{lem-arc}
For $f\in L^\infty(\T; \R)$ and $0<t\le\pi$ we have 
\begin{equation}\label{eq-arc}
\|A_tf- \widehat f(0) \|_3\le\sinc t\,\|f\|_\infty .
\end{equation}
\end{lemma}

\begin{proof}
Let $a = \widehat f(0)$ and normalize $\|f\|_\infty\le1$. The case $t=\pi$ is immediate since
$A_\pi f=a$. For $0<t\le\pi/2$, the function $h=A_tf$ is Lipschitz, with
\[
|h|\le1,\qquad
h'(\theta)=\frac{f(\theta+t)-f(\theta-t)}{2t}\quad\text{a.e.},
\qquad |h'|\le\frac1t.
\]
Since $\widehat h(0)=a$, Lemma~\ref{lem-lipschitz} gives
$\|A_tf-a\|_3^3\le1-\frac{3t}{2\pi}$, so~\eqref{eq-arc} follows from
\begin{equation}\label{eq-sinc-cubic}
1-\frac{3t}{2\pi}\le\cos^4\frac t2\le(\sinc t)^3,
\qquad 0\le t\le\frac\pi2 .
\end{equation}
For the second inequality in~\eqref{eq-sinc-cubic},
logarithmic differentiation gives
\[
\frac{d}{dt}\log\frac{(\sinc t)^3}{\cos^4(t/2)}
=\frac{2+\cos t}{\sin t}-\frac3t=\frac{t(2+\cos t)-3\sin t}{t\sin t}\ge0,
\]
because the numerator vanishes to third order at $t=0$ and its third
derivative is $t\sin t\ge0$; as the ratio tends to $1$ at $t=0$, it is at
least $1$. For the first inequality, $H(t)=\cos^4\frac t2-1+\frac{3t}{2\pi}$
vanishes at $t=0$ and at $t=\pi/2$, and
$2H'(t)=3/\pi - \sin t(1+\cos t)$. The subtracted term increases
on $[0,\pi/3]$ and decreases on $[\pi/3,\pi/2]$, ending at
$1>3/\pi$. Hence $H'$ changes sign exactly once, from positive to negative, and $H\ge0$ on $[0,\pi/2]$.

For $\pi/2<t<\pi$, splitting the circle at $\pm t$ and
substituting $\sigma\mapsto\pi+\sigma$ on the complementary arc gives
\[
t\,A_tf(\theta)+(\pi-t)A_{\pi-t}f(\theta+\pi)=\pi a,
\quad\text{hence}\quad
t(A_tf-a)=-(\pi-t)\left(A_{\pi-t}f(\cdot+\pi)-a\right).
\]
Since $0<\pi-t<\pi/2$, the case already proved gives
\[
\|A_tf-a\|_3\le\frac{\pi-t}{t}\,\sinc(\pi-t)=\sinc t .
\qedhere
\]
\end{proof}

\begin{lemma}\label{lem-mixture}
For $0<r<1$ let $w_r(t)=-\frac t\pi K_r'(t)$, $0<t<\pi$. Then $w_r\ge0$,
\begin{equation}\label{eq-mixture}
K_r=K_r(\pi)+\int_0^\pi w_r(t)\chi_t\,dt,
\end{equation}
and consequently
\begin{equation}\label{eq-mixture-moments}
K_r(\pi)+\int_0^\pi w_r(t)\,dt=1,
\qquad
\int_0^\pi w_r(t) \sinc t\,dt=r .
\end{equation}
\end{lemma}

\begin{proof}
The kernel $K_r$ is even and decreasing on $[0,\pi]$, so $w_r\ge0$, and for
$|\theta|\le\pi$,
\[
\int_0^\pi w_r(t)\chi_t(\theta)\,dt
=\int_{|\theta|}^\pi\left(-\frac t\pi K_r'(t)\right)\frac\pi t\,dt
=K_r(|\theta|)-K_r(\pi),
\]
which is~\eqref{eq-mixture}. The two identities
in~\eqref{eq-mixture-moments} are the Fourier coefficients
of~\eqref{eq-mixture} at $n=0$ and $n=1$, computed from
$\widehat{\chi_t}(0)=1$, $\widehat{\chi_t}(1)=\sinc t$ and
$\widehat{K_r}(0)=1$, $\widehat{K_r}(1)=r$.
\end{proof}

\begin{proof}[Proof of Theorem~\ref{thm-centered}]
Let $f$ be real with $\|f\|_\infty\le1$ and $a=\widehat f(0)$. Convolving
\eqref{eq-mixture} with $f$ and using the first identity
in~\eqref{eq-mixture-moments} to absorb the constants,
\[
\widetilde P_rf=P_rf-a=\int_0^\pi w_r(t)\left(A_tf-a\right)dt .
\]
Minkowski's integral inequality, Lemma~\ref{lem-arc} and the second identity
in~\eqref{eq-mixture-moments} give
\[
\|\widetilde P_rf\|_3\le\int_0^\pi w_r(t)\|A_tf-a\|_3\,dt
\le\int_0^\pi w_r(t) \sinc t\,dt=r .
\]
Smaller exponents follow because $L^p$ norms increase with $p$.

For sharpness, let $f=\pm1$ on sets of measure $\frac{1\pm a}2$, so that
$\widehat f(0)=a$ and
\[
\|f-a\|_p^p=\frac{(1+a)(1-a)^p+(1-a)(1+a)^p}2
=1+\frac{p(p-3)}2a^2+O(a^4),\qquad a\to0 .
\]
For $p>3$ this exceeds $1$ for all small $a\ne0$. Since $P_rf\to f$ in
$L^p$ as $r\uparrow1$, we get $\widetilde P_rf\to f-a$, so the estimate
$\|\widetilde P_rf\|_p\le r\|f\|_\infty$ fails for all $r$ close to $1$.
\end{proof}

\section{Complex functions with zero mean}\label{sec:complex}

Let $\Pp$ be the Riesz projection onto nonnegative frequencies, including
the constant term: $\widehat{\Pp f}(n)=\mathbf 1_{\{n\ge0\}}\widehat f(n)$.
We use two known bounds,
\begin{equation}\label{eq-Riesz-inputs}
\|\Pp f\|_4\le\|f\|_\infty,
\qquad
\|\Pp g\|_3\le\|g\|_4 ,
\end{equation}
both for complex-valued functions. The first is due to Marzo and
Seip~\cite[Theorem 1]{MarzoSeip}, who showed that $\|\Pp
f\|_{p}\le\|f\|_{\infty}$ holds precisely when $p\le4$. The second is the
case $q=4$ of $\Pp$ being a contraction from $L^q$ to 
$L^{4(1-1/q)}$, as was recently proved by Fu
and Li~\cite{FuLi}. The latter contraction property was conjectured by Brevig, Ortega-Cerd\`a,
Seip and Zhao~\cite[Conjecture 4]{BOSZ}; see~\cite{BLS} for an overview of the subject.  

\begin{proposition}\label{prop-interval}
For integers $a\le b$ let $\Pi_{[a,b]}f=\sum_{n=a}^b\widehat f(n)\zeta^n$.
Then
\begin{equation}\label{eq-interval-bound}
\|\Pi_{[a,b]}f\|_p\le\|f\|_\infty,
\qquad f\in L^\infty(\T),\ 0<p\le3,
\end{equation}
and the operator norm is exactly one.
\end{proposition}

\begin{proof}
The modulations $M_kf(\zeta)=\zeta^kf(\zeta)$ and the reflection
$Jf(\zeta)=f(\overline\zeta)$ are isometries of every $L^p(\T)$, and they
conjugate $\Pp$ into the one-sided projections
\[
Q_{\ge a}=M_a\Pp M_{-a},\qquad Q_{\le b}=M_bJ\Pp JM_{-b},
\]
whose multipliers are $\mathbf 1_{\{n\ge a\}}$ and $\mathbf 1_{\{n\le b\}}$.
Both estimates in~\eqref{eq-Riesz-inputs} therefore hold for $Q_{\ge a}$ and
$Q_{\le b}$, and comparing multipliers gives
$\Pi_{[a,b]}=Q_{\le b}Q_{\ge a}$. Applying the first estimate and then the
second,
\[
\|\Pi_{[a,b]}f\|_3\le\|Q_{\ge a}f\|_4\le\|f\|_\infty .
\]
All operators may be defined first on $L^2$, where the multiplier identity is
immediate, and $L^\infty\subset L^2$. Smaller exponents follow from
monotonicity of $L^p$ norms, and $f(\zeta)=\zeta^a$ gives equality.
\end{proof}

In particular, the symmetric partial sums $S_N=\Pi_{[-N,N]}$ are contractions from $L^\infty(\T)$ to $L^3$.

\begin{proof}[Proof of Theorem~\ref{thm-complex}]
We will prove a more general inequality: for all $f\in L^\infty(\T)$
\begin{equation}\label{eq-complex-correction}
\big\|P_rf-(1-r)\widehat f(0)\big\|_3 \le r\|f\|_\infty .
\end{equation}
Since $(1-r)\sum_{N\ge|n|}r^N=r^{|n|}$, comparing multipliers on $L^2$
gives the classical representation of $P_r$ as an Abel mean of partial sums,
$P_r=(1-r)\sum_{N\ge0}r^NS_N$. The term $N=0$ is $(1-r)\widehat f(0)$, so
\begin{equation}\label{eq-abel-mean}
P_rf-(1-r)\widehat f(0)=(1-r)\sum_{N=1}^\infty r^NS_Nf .
\end{equation}
By Proposition~\ref{prop-interval} each $S_N$ is a contraction from
$L^\infty$ to $L^3$, so the series converges absolutely in $L^3$, and
\[
\big\|P_rf-(1-r)\widehat f(0)\big\|_3
\le(1-r)\sum_{N=1}^\infty r^N\|f\|_\infty=r\|f\|_\infty ,
\]
which is~\eqref{eq-complex-correction}. If $\widehat f(0)=0$ this says
$\|P_rf\|_3\le r\|f\|_\infty$; since $P_r\zeta=r\zeta$, the norm of $P_r$
from $L^\infty_0(\T)$ to $L^p$, $p\le3$, is exactly $r$.
\end{proof}

\begin{remark}\label{rem-decreasing}
More generally, let $(\lambda_n)_{n\ge0}$ be nonnegative and nonincreasing,
with limit $\ell$, and let $\widehat{\Lambda f}(n)=\lambda_{|n|}\widehat f(n)$.
Summation by parts gives
$\Lambda=\ell I+\sum_{N\ge0}(\lambda_N-\lambda_{N+1})S_N$, a combination
with nonnegative coefficients of total weight $\lambda_0$, so the same
argument shows $\|\Lambda f\|_3\le\lambda_0\|f\|_\infty$. The Poisson
operator is the case $\lambda_n=r^n$.
\end{remark}
 
\begin{remark}\label{rem-interval-sharpness}
The exponent $3$ in Proposition~\ref{prop-interval} is optimal already for
$S_1$, and it is also optimal in~\eqref{eq-complex-correction}. Indeed, let
$f_\varepsilon(e^{i\theta})=e^{i\varepsilon\cos\theta}$, $\varepsilon > 0$. 
Expanding the exponential and retaining the frequencies $-1,0,1$ gives
$S_1f_\varepsilon=c_\varepsilon+i d_\varepsilon\cos\theta$, where
\[
c_\varepsilon=1-\frac{\varepsilon^2}{4}
+\frac{\varepsilon^4}{64}+O(\varepsilon^6),
\qquad
d_\varepsilon=\varepsilon-\frac{\varepsilon^3}{8}+O(\varepsilon^5).
\]
It follows that 
\begin{equation}\label{eq-interval-obstruction}
\int_\T|S_1f_\varepsilon|^p\dd
=1+\frac{p(p-3)}{64}\varepsilon^4+O_p(\varepsilon^6).
\end{equation}
For every $p>3$ this exceeds $1$ when $\varepsilon>0$ is sufficiently
small. Since, by~\eqref{eq-abel-mean},
\[
r^{-1}\bigl(P_rf_\varepsilon-(1-r)\widehat{f_\varepsilon}(0)\bigr)
=(1-r)\sum_{N\ge1}r^{N-1}S_Nf_\varepsilon\to S_1f_\varepsilon
\]
uniformly as $r\downarrow0$, the inequality
\eqref{eq-complex-correction} also fails for every $p>3$.
\end{remark}

We conclude the section with the upper bound in~\eqref{eq-pc-bounds}. The
underlying obstruction, in the limit $r\to0$, is due to Brevig,
Ortega-Cerd\`a and Seip~\cite{BOS}; we include the short proof for
completeness.

\begin{proposition}\label{prop-complex-obstruction}
For every $p>4$ there exist $f\in L^\infty_0(\T)$ with $|f|=1$ and
$0<r<1$ such that $\|P_rf\|_p>r$.
\end{proposition}

\begin{proof}
For $0<\varepsilon<1$ put
\[
f_\varepsilon(\zeta)=\zeta\,
\frac{1-\varepsilon\overline\zeta^{\,2}}{1-\varepsilon\zeta^2},
\qquad\zeta\in\T .
\]
Then $|f_\varepsilon|=1$, and $f_\varepsilon(-\zeta)=-f_\varepsilon(\zeta)$,
so only odd frequencies occur and $\widehat{f_\varepsilon}(0)=0$. Expanding
the geometric series,
\[
f_\varepsilon(\zeta)=-\varepsilon\overline\zeta
+(1-\varepsilon^2)\sum_{k\ge0}\varepsilon^k\zeta^{2k+1},
\]
an absolutely convergent series, so that $P_rf_\varepsilon/r\to
g_\varepsilon:=(1-\varepsilon^2)\zeta-\varepsilon\overline\zeta$ uniformly as
$r\downarrow0$. Factoring out $(1-\varepsilon^2)\zeta$ and using
$\int_\T|1-\tau\zeta|^p\dd=1+\frac{p^2}4\tau^2+O_p(\tau^4)$ with
$\tau=\varepsilon/(1-\varepsilon^2)$,
\begin{equation}\label{eq-obstruction-expansion}
\int_\T|g_\varepsilon|^p\dd=1+\frac{p(p-4)}4\varepsilon^2+O_p(\varepsilon^4).
\end{equation}
For $p>4$ choose $\varepsilon$ small with $\int_\T|g_\varepsilon|^p\dd>1$,
and then $r$ small enough that $\|P_rf_\varepsilon\|_p>r$.
\end{proof}

\section{Conjectures}\label{sec:questions}

Theorem~\ref{thm-complex} and
Proposition~\ref{prop-complex-obstruction} give $3\le p_\C\le4$, and we
expect the upper bound to be the truth.

\begin{conjecture}\label{conj-complex}
The Poisson operator is Schwarz-contractive from $L^\infty_0(\T)$ to
$L^4(\T)$; that is, $\int_\T|P_rf|^4\dd\le r^4$ whenever $\|f\|_\infty\le1$
and $\widehat f(0)=0$.
\end{conjecture}

In~\cite{KovalevYang}, Yang and the author proved the limiting case $r\to 0$ of Conjecture~\ref{conj-complex}. Brevig, Ortega-Cerd\`a and Seip~\cite{BOS} showed that even in this infinitesimal version, the exponent $4$ cannot be replaced by a larger one. 

Although the proof of Theorem~\ref{thm-centered} breaks down for complex-valued functions, numerical evidence suggests that its statement, as well as the key estimate in Lemma~\ref{lem-arc}, remain true.

\begin{conjecture}\label{conj-complex-centered}
The centered Poisson operator $\widetilde P_rf=P_rf-\widehat f(0)$ is Schwarz-contractive from $L^\infty(\T)$ to
$L^3(\T)$; that is, $\int_\T|\widetilde P_rf |^3\dd\le r^3$ whenever $\|f\|_\infty\le1$.
\end{conjecture}

\section{AI use disclosure}

The author acknowledges the use of AI tools, specifically ChatGPT by OpenAI, Claude by Anthropic, and Gemini by Google.

\end{document}